\documentclass[final,leqno]{siamltex}
\usepackage{amsmath}
\usepackage{graphicx}
\usepackage[notcite,notref]{showkeys}
\usepackage{mathrsfs}
\usepackage{graphics}
\usepackage{float}
\usepackage{amsfonts,amssymb}
\usepackage{dsfont}
\usepackage{pifont}
\usepackage{wrapfig} 
\usepackage{hyperref}
\usepackage{multirow}
\usepackage{color}
\numberwithin{equation}{section}
\usepackage{threeparttable}
\def\3bar{{|\hspace{-.02in}|\hspace{-.02in}|}}
\def\E{{\mathcal{E}}}
\def\T{{\mathcal{T}}}

\def\bpsi{\boldsymbol{\psi}}

\def\pT{{\partial T}}

\def\W{{\mathcal{W}}}

\def\bw{{\mathbf{w}}}
\def\bu{{\mathbf{u}}}
\def\bg{{\mathbf{g}}}
\def\bv{{\mathbf{v}}}
\def\bn{{\mathbf{n}}}
\def\bQ{{\mathbf{Q}}}

\def\bq{{\mathbf{q}}}
\def\be{{\mathbf{e}}}
\def\bw{{\mathbf{w}}}

\def\bf{{\mathbf{f}}}
\def\bphi{{\boldsymbol{\phi}}}
\def\bsigma{{\boldsymbol{\sigma}}}
\def\bepsilon{{\boldsymbol{\epsilon}}}

\def\bphi{{\boldsymbol{\phi}}}
\def\ljump{{[\![}}
\def\rjump{{]\!]}}

\def\bsigma{{\boldsymbol{\sigma}}}

\newtheorem{algorithm}{Weak Galerkin Algorithm}[section]

\title {A Weak Galerkin Method for Elasticity Interface Problems}

\begin{document}

\author{
Chunmei Wang \thanks{Department of Mathematics, University of Florida, Gainesville, FL 32611, USA (chunmei.wang@ufl.edu). The research of Chunmei Wang was partially supported by National Science Foundation Award DMS-2136380.}
 \and  
 Shangyou Zhang\thanks{Department of Mathematical Sciences,  University of Delaware, Newark, DE 19716, USA (szhang@udel.edu).  }}

\maketitle

\begin{abstract}
 This article introduces a weak Galerkin (WG) finite element method for linear elasticity interface problems on general polygonal/ployhedra partitions. The developed WG method has been proved to be stable and accurate with optimal order error estimates in the discrete $H^1$ norm. Some numerical experiments are conducted to verify the efficiency and accuracy of the proposed WG method.
\end{abstract}



\pagestyle{myheadings}

\section{Introduction}
 In this paper we are concerned with the development of weak Galerkin (WG) finite element methods for elasticity interface problems. To this end, assume $\Omega\subset \mathbb R^d$ ($d = 2, 3$) is an open bounded domain with piecewise smooth Lipschitz boundary $\partial \Omega$. Let $N$ and $M$ be two positive integers. The domain $\Omega$ is partitioned into a set of subdomains $\{\Omega_i\}_{i=1}^N$ with piecewise smooth Lipschitz
boundary $\partial \Omega_i$ for $i=1, \cdots, N$; $\Gamma=\bigcup_{i=1}^N \partial \Omega_i\setminus \partial \Omega$ is the interface between the subdomains in the sense that
$$
\Gamma =\bigcup_{m=1}^{M} \Gamma_m,
$$
where there exist $i, j \in \{1, \cdots, N\}$ such that $\Gamma_m=\partial \Omega_i \cap\partial \Omega_j$ for $m=1, \cdots, M$.
 We consider the following elasticity interface problem: Find the displacement $\bu$ such that
\begin{equation}\label{model}
\begin{split}
\bsigma(\bu_i)&=2\mu_i\bepsilon(\bu_i)+\lambda_i\nabla \cdot \bu_i I, \qquad \text{in}\ \Omega_i, i=1, \cdots, N,\\
-\nabla \cdot \bsigma(\bu_i) &=\bf_i, \qquad\qquad\qquad\qquad\qquad \text{in}\ \Omega_i,i=1, \cdots, N,\\
\ljump \bu\rjump_{\Gamma_m}&=\bphi_m, \qquad\qquad\qquad\qquad\qquad \text{on}\ \Gamma_m, m=1,\cdots, M,\\
\ljump \bsigma(\bu)\bn\rjump_{\Gamma_m}&=\bpsi_m, \qquad\qquad\qquad\qquad\qquad \text{on}\ \Gamma_m,m=1,\cdots, M,\\
\bu_i&=\bg_i,\quad\qquad \qquad\qquad\qquad\qquad \text{on}\ \partial\Omega_i\cap \partial\Omega, i=1,\cdots, N,
\end{split}
\end{equation}
where $\bu_i=\bu|_{\Omega_i}$, $\bf_i=\bf|_{\Omega_i}$, $\bg_i=\bg|_{\Omega_i}$, $\mu_i=\mu|_{\Omega_i}$, $\lambda_i=\lambda|_{\Omega_i}$,   $[\![\bsigma(\bu)  \bn]\!]_{\Gamma_m}=\bsigma(\bu_i)  \bn_i+ \bsigma(\bu_j)  \bn_j$ with $\bn_i$ and $\bn_j$ being the unit outward normal directions to $\partial \Omega_i\cap \Gamma_m$ and $\partial \Omega_j\cap \Gamma_m$, and $[\![\bu]\!]_{\Gamma_m}=\bu_i|_{  \Gamma_m}-\bu_j|_{\Gamma_m}$. 
Throughout this paper, we use bold face letters to denote vector valued functions and their associated function spaces. Here, $\bepsilon=\frac{1}{2} (\nabla \bu+\nabla\bu^T)$ is the strain tensor; $\bf$ represents a given body force; and $\lambda$ and $\mu$ are the positive Lam$\acute{e}$ parameters. Regarding to the Young’s modulus $E$, and the Poisson's ratio $\nu$,   the following identities hold true; i.e.,
$$
\lambda=\frac{E\nu}{(1-2\nu)(1+\nu)}, \qquad \mu=\frac{E}{2(1+\nu)}.
$$
The Lam$\acute{e}$ parameters $\lambda$ and $\mu$ are assumed to be piecewise smooth functions with respect to the partition $\Omega=\bigcup_{i=1}^N \Omega_i$. 
We assume that $\frac{\lambda}{\mu}=\frac{2\nu}{1-2\nu}$ is bounded.


Elasticity interface problems play an important role in continuum mechanics where the elasticity theory and the governing partial differential equations (PDEs) describe various material behaviors.  An interface description for this class of problem in the elasticity theory is indispensable whenever there are voids, pores, inclusions, dislocations, cracks or composite structures in materials \cite{w1, w2, w3, w4}. In particular, the elasticity interface problems are crucial in tissue engineering, biomedical science and biophysics  \cite{w5, w6, w7}. In many situations, the interface is not static such as fluid–structure interfacial boundaries \cite{w8}. Discontinuities in material properties often occur over the interface.  
When an elastic body is occupied by heterogeneous materials with distinct physical parameters, the governing equation holds on each disjoint domain. The solution to the governing equation is required to satisfy the displacement and traction jump conditions along the interface between different materials besides the usual boundary conditions.
In the linear elasticity theory, the stress–strain relation is governed by the constitutive equations. For the isotropic homogeneous materials, constitutive equations are determined by any two terms of bulk modulus, Young’s modulus, Lam$\acute{e}$’s first parameter, shear modulus, Poisson’s ratio, and P-wave modulus  \cite{w9}. For these moduli being position dependent functions, the related constitutive equations describe elasticity property of isotropic inhomogeneous media. In seismic wave equations, inhomogeneity is accounted by assuming Lam$\acute{e}$'s parameters to be a position dependent function  \cite{w10}.

Many numerical methods have been designed for elasticity interface problems. The immersed interface method (IIM)  \cite{w15}  was proposed to solve elasticity interface problems for isotropic homogeneous media \cite{w16, w17}. A second-order sharp numerical method was developed for linear elasticity equations  \cite{w18}. Finite element methods including the partition of unity method (PUM), the generalized finite element method (GFEM) and the extended finite element method (XFEM) were developed to capture the non-smooth property of the solution over the interface by adding enrichment functions to the approximation  \cite{w2, w3, w4}. The discontinuous Galerkin finite element methods were employed to simulate strong and weak discontinuities  \cite{w19, w20, w21}  through the weak enforcement of the continuity. The immersed finite element method (IFM) was proposed to solve elasticity problems with inhomogeneous jump conditions  \cite{w22, w23, w24}. The sharp-edged interface was proposed for a special elasticity interface problem  \cite{w25}. The  bilinear IFM was introduced and further modified to a locking-free version \cite{w26, w27}. The immersed meshfree Galerkin method was proposed for composite solids \cite{w28}. A Nitsche type method was proposed for elasticity interface problems \cite{w29}.  
 
The weak formulation for the elasticity interface model problem \eqref{model} is as follows: Find $\bu$ satisfying $\bu_i=\bg_i$ on $\partial\Omega_i\cap\partial \Omega$ ($i=1, \cdots, N$), such that
\begin{equation}\label{weakform}
(2\mu\epsilon(\bu), \epsilon(\bv))+(\lambda \nabla\cdot\bu, \nabla\cdot\bv)=(\bf, \bv)+\sum_{m=1}^M\langle\bpsi,\bv\rangle_{\Gamma_m}, \quad\forall \bv\in [H_0^1(\Omega)]^d.
\end{equation}

We follow the standard notations for Sobolev spaces
and norms defined on a given open and bounded domain $D\subset \mathbb{R}^d$ with
Lipschitz continuous boundary. As such, $\|\cdot\|_{s,D}$ and
$|\cdot|_{s,D}$ are used to denote the norm and seminorm in the Sobolev space
$H^s(D)$ for any $s\ge 0$. The inner product in $H^s(D)$
is denoted by $(\cdot,\cdot)_{s,D}$ for $s\ge 0$. The space $H^0(D)$ coincides with
$L^2(D)$ (i.e., the space of square integrable functions), for which the norm and the inner product are denoted as $\|\cdot \|_{D}$ and $(\cdot,\cdot)_{D}$.  When
$D=\Omega$ or when the domain of integration is clear from the
context, we shall drop the subscript $D$ in the norm and the inner product
notation.

The paper is organized as follows. In Section 2, we briefly review the weak differential operators and their discrete analogies. In Section 3,  the WG method for the model problem \eqref{model} based on the weak formulation (\ref{weakform}) is proposed. In Section 4, we establish the solution existence, uniqueness, and stability. In Section 5, we derive an error equation for the WG solutions. The optimal order error estimate for the exact solution in the discrete $H^1$ norm is established in Section 6. Finally, a couple of numerical results to illustrated and verify our convergence theory are reported in Section 7.

 

\section{Weak Differential Operators}\label{Section:Hessian}
The two principal differential operators in the weak formulation (\ref{weakform}) for the elasticity interface model problem (\ref{model}) are the divergence operator $(\nabla\cdot)$ and the gradient operator $\nabla$.  The discrete weak versions for $(\nabla\cdot)$ and $\nabla$ have been introduced in \cite{welasticity, wy3655}. For completeness, we shall briefly review their definitions in this section.

Let $T$ be a polygonal or polyhedral domain with boundary $\partial T$. A vector-valued weak function on $T$ refers to  $\bv=\{\bv_0,\bv_b\}$ with $\bv_0\in [L^2(T)]^d$ and $\bv_b\in [L^{2}(\partial T)]^d$. Here $\bv_0$ and $\bv_b$ are used to represent the values of $\bv$ in the interior and on the boundary of $T$ respectively. Note that $\bv_b$ may not necessarily be the trace of $\bv_0$ on $\partial T$. Denote by $\W(T)$ the space of weak functions on $T$; i.e.,
\begin{equation}\label{2.1}
\W(T)=\{\bv=\{\bv_0,\bv_b\}: \bv_0\in [L^2(T)]^d, \bv_b\in
[L^{2}(\partial T)]^d\}.
\end{equation}

The weak divergence of $\bv\in \W(T)$, denoted by $\nabla_w \cdot  \bv$, is defined as a linear functional on $H^1(T)$ such that
\begin{equation*}
(\nabla_w \cdot \bv, w)_T= -(\bv_0, \nabla w)_T+\langle \bv_b, w \textbf{n}\rangle_{\partial T},
\end{equation*}
for all $w \in  H^1(T)$, where $\bn$ is an unit outward normal direction to $\partial T$.

The weak gradient of $\bv\in \W(T)$, denoted by $\nabla_w \bv$, is defined as a linear functional on $[H^1(T)]^{d\times d}$ such that
\begin{equation*}
(\nabla_w  \bv, \bw)_T=-(\bv_0,\nabla \cdot \bw)_T+\langle \bv_b,\bw  \textbf{n}\rangle_{\partial T},
\end{equation*}
for all $\bw\in [H^1(T)]^{d\times d}$.

Denote by $P_r(T)$ the space of polynomials on $T$ with degree no more than $r$. A discrete version of $\nabla_w \cdot \bv$ for $\bv\in \W(T)$, denoted by $\nabla_{w, r, T}
\cdot \bv$, is defined as the unique polynomial in $P_r(T)$ satisfying
\begin{equation}\label{divgence}
(\nabla_{w, r, T}\cdot \bv, w)_T= -(\bv_0, \nabla w)_T+\langle \bv_b, w \textbf{n}\rangle_{\partial T},\quad\forall w \in P_r(T).
\end{equation}

A discrete version of $\nabla_{w}\bv$  for $\bv\in \W(T)$, denoted by $\nabla_{w, r, T}\bv$, is defined as a unique polynomial matrix in $[P_r(T) ]^{d\times d}$ satisfying
\begin{equation}\label{disgradient}
(\nabla_{w, r, T}\bv, \bw)_T=-(\bv_0,\nabla \cdot \bw)_T+\langle \bv_b,\bw  \textbf{n}\rangle_{\partial T}, \quad\forall\bw\in [P_r(T)]^{d\times d}.
\end{equation}
Using the weak gradient, we define the weak strain tensor as follows:
$$
\epsilon_w(\bv)=\frac{1}{2}(\nabla_w\bv+\nabla_w\bv^T).
$$
A discrete version of $\epsilon_{w}(\bv)$  for $\bv\in \W(T)$, denoted by $\epsilon_{w, r, T}(\bv)$, is defined as a unique polynomial matrix in $[P_r(T) ]^{d\times d}$ satisfying
\begin{equation}\label{disstrain}
(\epsilon_{w, r, T}(\bv), \bw)_T=-(\bv_0,\epsilon(\bw))_T+\langle \bv_b,\bw  \textbf{n}\rangle_{\partial T}, 
\end{equation}
for any symmetric matrix $\bw\in [P_r(T)]^{d\times d}$.
 

\section{Weak Galerkin Algorithm}\label{Section:WGFEM}
Let ${\cal T}_h$ be a finite element partition of the domain $\Omega$ consisting of polygons or polyhedra that are shape-regular \cite{wy3655}. Assume that the edges/faces of the elements in ${\cal T}_h$ align with the interface $\Gamma$. The partition ${\cal T}_h$ can be grouped into $N$ sets of elements denoted by ${\cal T}_h^{i}={\cal T}_h \cap \Omega_i$, so that each ${\cal T}_h^i$ provides a finite element partition for the subdomain $\Omega_i$ for $i=1, \cdots, N$.
The intersection of the partition ${\cal T}_h$ also introduces a finite element partition for the interface $\Gamma$, denoted by $\Gamma_h$.
Denote by ${\mathcal E}_h$ the
set of all edges or flat faces in ${\cal T}_h$ and  ${\mathcal
E}_h^0={\mathcal E}_h \setminus
\partial\Omega$ the set of all interior edges or flat faces.
Denote by $h_T$ the meshsize of $T\in {\cal T}_h$ and
$h=\max_{T\in {\cal T}_h}h_T$ the meshsize for the partition
${\cal T}_h$.

On each element $T\in {\cal T}_h$, denote by $RM(T)$ the space of rigid motions on $T$ given by
$$
RM(T)=\{\textbf{a}+\eta \textbf{x}: \textbf{a}\in \mathbb R^d, \eta\in so(d)\},
$$
where $\textbf{x}$ is the position vector on $T$ and $so(d)$ is the space of skew-symmetric $d\times d$ matrices. The trace of the rigid motion on each edge $e\subset T$ forms a finite dimensional space denoted by  $P_{RM} (e)$; i.e.,
$$
P_{RM} (e)=\{ \bv\in [L^2(e)]^d: \bv=\tilde{\bv}|_e \ \text{for some}\ \tilde{\bv}\in RM(T), e\subset \partial T\}.
$$

For any integer $k\geq 1$, denote by
$W_k(T)$ the local discrete space of  the weak functions given by
$$
W_k(T)=\{\{\bv_0,\bv_b\}:\bv_0\in [P_k(T)]^d, \bv_b\in
S_k(e), e\subset \partial T\},
$$
where $S_k(e)= [P_{k-1}(e)]^d+P_{RM} (e)$. Since $P_{RM} (e)\subset P_1(e)$, then the boundary component $S_k(e)$ is given by $[P_{k-1} (e)]^d$ for $k>1$ and $P_{RM} (e)$ for $k=1$. 

Patching $W_k(T)$ over all the elements $T\in {\cal T}_h$
through a common value $\bv_b$ on the interior interface $\E_h^0$ (excluding the interface $\Gamma$), we arrive at the following
weak finite element space $W_h$; i.e.,
$$
W_h=\big\{\{\bv_0, \bv_b\}:\{\bv_0, \bv_b\}|_T\in W_k(T), \forall T\in {\cal T}_h \big\}.
$$
Denote by $W_h^0$ the subspace of $W_h$ with homogeneous boundary values; i.e.,
\begin{equation*}\label{wh0}
 W_h^0=\{\{\bv_0,\bv_b\} \in W_h:\  \bv_b|_{\partial\Omega}=0 \}.
\end{equation*}

For simplicity of notation and without confusion, for any $\bv\in
W_h$, denote by $\nabla_w\cdot\bv$, $\nabla_{w}\bv$ and $\epsilon_{w}(\bv)$ the discrete weak actions $\nabla_{w, k-1, T}\cdot\bv$, $\nabla_{w, k-1, T}\bv$ and $\epsilon_{w, k-1, T}(\bv)$  computed by using (\ref{divgence}), (\ref{disgradient}) and \eqref{disstrain} on each element $T$ respectively; i.e.,
 $$
\nabla_w\cdot\bv|_T=\nabla_{w, k-1, T}\cdot(\bv|_T), \qquad \bv\in W_h,
$$
$$
(\nabla_{w}\bv)|_T= \nabla_{w, k-1, T}(\bv|_T), \qquad \bv\in W_h,
$$
$$
(\epsilon_{w, r, T} (\bv)) |_T=\epsilon_{w, k-1, T}(\bv|_T) , \qquad \bv\in W_h.
$$
Denote by $\bQ_b$ the $L^2$ projection onto the space $S_k(e)$.
For any $\bu, \bv\in W_h$, we introduce the
following bilinear forms
\begin{align} \label{EQ:local-stabilizer}
s(\bu, \bv)=&\sum_{T\in {\cal T}_h}s_T(\bu, \bv),
\\
a(\bu, \bv)=&\sum_{T\in {\cal T}_h}a_T(\bu, \bv),  \label{EQ:local-bterm}
\end{align}
where
\begin{equation*}
\begin{split}
s_T(\bu, \bv)=&h_T^{-1}\langle \bQ_b\bu_0-\bu_b, \bQ_b\bv_0-\bv_b\rangle_{\partial T},\\
a_T(\bu, \bv)=&(2\mu \epsilon_w(\bu), \epsilon_w(\bv))_T+(\lambda \nabla_w\cdot\bu, \nabla_w\cdot\bv)_T.
\end{split}
\end{equation*}

The following is the weak Galerkin scheme for the elasticity interface problem (\ref{model}) based on the variational formulation (\ref{weakform}).
\begin{algorithm}\label{WG}
Find $\bu_h \in  W_{h}$, such that  
$\bu_b=\bQ_b\bg_i$ on $\partial\Omega_i\cap\partial\Omega$ for $i=1, \cdots, N$ and $\bu_b^L-\bu_b^R=\bQ_b\bphi_m$ on $\Gamma_m$ for $m=1,\cdots, M$, s.t.
\begin{equation}\label{32}
s(\bu, \bv)+a(\bu, \bv)=(\bf, \bv_0)+\sum_{m=1}^M\langle \bpsi, \bv_b\rangle_{\Gamma_m},  
\end{equation}
for any $\bv\in W_h^0$. Here $\bu_b^L$ and $\bu_b^R$ denote $\bu_b|_{\Omega_i}$ and   $\bu_b|_{\Omega_j}$ if $\Gamma=\partial\Omega_i\cap\partial\Omega_j$ for some $i$ and $j$. 
\end{algorithm}

\section{Stability Analysis}\label{Section:ExistenceUniqueness}

\begin{lemma}\cite{welasticity} (Second Korn's Inequality)
Assume that the domain $\Omega$ is connected, open bounded with Lipschitz continuous boundary. Let $\Xi\subset\partial\Omega$ be a nontrivial portion of the boundary $\partial\Omega$ with dimension $d-1$. For any fixed real number $1\leq p<\infty$, there exists a constant $C$ such that
\begin{equation}\label{k}
\|\bv\|_1\leq C(\|\epsilon (\bv)\|_0+\|\bv\|_{L^p(\Xi)}),
\end{equation}
for any $\bv\in [H^1(\Omega)]^d$.
\end{lemma}

\begin{theorem}
There exists a unique solution to the weak Galerkin finite element scheme \eqref{32}.
\end{theorem}
\begin{proof}
Since the number of equations is equal to the number of unknowns in \eqref{32}, it suffices to prove the solution uniqueness. To this end, let $\bu_h^{(1)}=\{\bu_0^{(1)}, \bu_b^{(1)}\}$ and $\bu_h^{(2)}=\{\bu_0^{(2)}, \bu_b^{(2)}\}\in W_h$ be two different solutions of \eqref{32}. It follows that $\bu_b^{(j)} (j=1, 2)$ satisfies $\bu_b^{(j)}=\bQ_b\bg_i$ on $\partial\Omega_i\cap\partial\Omega$ for $i=1, \cdots, N$ and $\bu_b^{(j), L}-\bu_b^{(j), R}=\bQ_b\bphi_m$ on $\Gamma_m$ for $m=1,\cdots, M$, and
\begin{equation}\label{t1}
s(\bu^{(j)}, \bv)+a(\bu^{(j)}, \bv)=(\bf, \bv_0)+\sum_{m=1}^M\langle \bpsi, \bv_b \rangle_{\Gamma_m},  \qquad \forall \bv=\{\bv_0, \bv_b\}\in W_h^0, j=1, 2.
\end{equation} 
The difference of the two solutions $\bw=\bu_h^{(1)}-\bu_h^{(2)} \in V_h^0$ satisfies 
\begin{equation}\label{t2}
s(\bw, \bv)+a(\bw, \bv)=0, \qquad \forall \bv=\{\bv_0, \bv_b\} \in W_h^0.
\end{equation} 
This implies that 
\begin{eqnarray}\label{c1}
\epsilon_w(\bw)&=&0, \qquad \text{in each}\ T\in {\cal T}_h,\\ 
\nabla_w\cdot\bw&=&0, \qquad \text{in each}\ T\in {\cal T}_h,\label{c2}\\
Q_b \bw_0&=&\bw_b,  \qquad \text{on each}\ \partial T.\label{c3}
\end{eqnarray}
Using \eqref{disgradient} and the usual integration by parts, we get
\begin{equation*}
\begin{split}
(\nabla_w \bw, \bq)_T&=-(  \bw_0, \nabla \cdot \bq)_T+\langle \bw_b,\bq  \bn\rangle_{\pT}\\
&=(\nabla \bw_0,\bq)_T-\langle \bw_0-\bw_b,\bq  \bn\rangle_{\pT}\\
&=(\nabla \bw_0,\bq)_T-\langle \bQ_b\bw_0-\bw_b,\bq  \bn\rangle_{\pT},\\
\end{split}
\end{equation*}
for any $\bq\in [P_{k-1}(T)]^{d\times d}$. Using \eqref{c3} gives $\nabla_w \bw=\nabla \bw_0$ on each element $T$. Using \eqref{c1}, we have $\epsilon (\bw_0)=\epsilon_w(\bw)=0$, which yields  $\bw_0\in RM(T)\subset [P_1(T)]^d$. 
Using $\bw_0|_e=\bQ_b\bw_0=\bw_b$ we obtain $\bw_0$ is a continuous function in $\Omega$ with vanishing boundary value on $\partial \Omega$. Using the second Korn's inequality \eqref{k} gives $\bw_0\equiv 0$ in $\Omega$, and further $\bw_b\equiv 0$ due to \eqref{c3}. Therefore, we have $\bu_h^{(1)}=\bu_h^{(2)}$. This completes the proof of the theorem. 
\end{proof}

\section{Error Equations}\label{Section:error-equation}
The goal of this section is to derive the error equation for the weak Galerkin method (\ref{32}). The error equation plays a crucial role in the convergence analysis discussed in Section 6.

Let $\bQ_0$ and $\bQ_b$ be the $L^2$ projection operators onto $[P_k(T)]^d$ and $S_k(e)$ respectively. For $\bw\in H^1(\Omega)$, define the $L^2$ projection $\bQ_h \bw\in W_h$ as follows
$$
\bQ_h\bw|_T=\{\bQ_0\bw, \bQ_b\bw\}.
$$
Denote by ${\cal Q}_h$ the $L^2$ projection operator onto the finite element space of piecewise polynomials of degree $k-1$.

For simplicity, we assume the coefficients $\mu$ and $\lambda$ are piecewise constants with respect to the finite element partition ${\cal T}_h$. Note that the analysis can be extended to piecewise smooth coefficients $\mu$ and $\lambda$ without any difficulty.

\begin{lemma}\label{Lemma5.1} \cite{welasticity} The operators $\bQ_h$ and ${\cal Q}_h$ satisfy the following commutative properties:
\begin{eqnarray}\label{l}
\epsilon_{w} (\bQ_h \bw) &=& {\cal Q}_h (\epsilon(\bw)), \qquad  \forall   \bw\in [H^1(T)]^d,\\
\nabla_{w}\cdot (\bQ_h \bw) &=& {\cal Q}_h(\nabla \cdot \bw),  \qquad  \forall  \bw\in [H^1(T)]^d. \label{l-2}
\end{eqnarray}
\end{lemma}
 
Let $\bu$ and $\bu_h \in  W_h$ be the exact solution of the elasticity  interface problem \eqref{model} and its numerical solution arising from the WG scheme (\ref{32}). Denote the error function by
\begin{equation}\label{error}
\be_h= {\cal Q}_h\bu-\bu_h.
\end{equation}

\begin{lemma}\label{errorequa}
Let $\bu$ and $\bu_h \in W_h$ be the exact solution of elasticity interface problem \eqref{model} and its numerical solution arising from the WG scheme (\ref{32}). The error function $\be_h$ defined in (\ref{error}) satisfies the following  equation:
\begin{equation}\label{sehv}
\begin{split}
&s(\be_h, \bv)+a(\be_h, \bv) =  s(\bQ_h\bu, \bv) \\&-\sum_{T\in {\cal T}_h} \langle \bv_0-\bv_b, 2\mu({\cal Q}_h(\epsilon(\bu))-\epsilon(\bu))\cdot\bn+\lambda ({\cal Q}_h (\nabla\cdot\bu)-\nabla\cdot\bu)\cdot\bn \rangle_{\pT}.
\end{split}
\end{equation}
\end{lemma}
 
\begin{proof}
Using \eqref{disstrain}, \eqref{l} and the usual integration by parts gives
 \begin{equation*}
 \begin{split}
 &\sum_{T\in {\cal T}_h}(2\mu\epsilon_w(\bQ_h\bu), \epsilon_w(\bv))_T\\
 =& \sum_{T\in {\cal T}_h}(2\mu {\cal Q}_h(\epsilon(\bu)), \epsilon_w(\bv))_T\\
 = &\sum_{T\in {\cal T}_h}-(2\mu \bv_0, \epsilon({\cal Q}_h(\epsilon(\bu))))_T+\langle 2\mu \bv_b, {\cal Q}_h(\epsilon(\bu))\cdot\bn \rangle_{\pT}
\\ =  &\sum_{T\in {\cal T}_h} (2\mu\epsilon(\bv_0),  {\cal Q}_h(\epsilon(\bu)))_T-\langle 2\mu(\bv_0-\bv_b), {\cal Q}_h(\epsilon(\bu))\cdot\bn \rangle_{\pT}
\\ =  &\sum_{T\in {\cal T}_h} (2\mu\epsilon(\bv_0),   \epsilon(\bu))_T-\langle 2\mu(\bv_0-\bv_b), {\cal Q}_h(\epsilon(\bu))\cdot\bn \rangle_{\pT} 
\\ =  &\sum_{T\in {\cal T}_h} (\bv_0, -\nabla\cdot(2\mu \epsilon(\bu)))_T-\langle 2\mu(\bv_0-\bv_b), ({\cal Q}_h(\epsilon(\bu))-\epsilon(\bu))\cdot\bn \rangle_{\pT}\\
  &\quad +\sum_{m=1}^M \langle \bv_b, \ljump2\mu\epsilon(\bu)\bn\rjump\rangle_{\Gamma_m},
\end{split}
\end{equation*}
where we used $\sum_{T\in {\cal T}_h} \langle \bv_b, \ljump 2\mu\epsilon(\bu)\bn\rjump\rangle_{\pT}=\sum_{m=1}^M \langle \bv_b, \ljump 2\mu\epsilon(\bu)\bn\rjump\rangle_{\Gamma_m} $ since $\bv_b=0$ on $\partial \Omega$ on the last line.
 
Using \eqref{divgence}, \eqref{l-2} and the usual integration by parts yields
 \begin{equation*}
 \begin{split}
&\sum_{T\in {\cal T}_h} (\lambda\nabla_w\cdot(\bQ_h\bu), \nabla_w\cdot\bv)_T\\
=& \sum_{T\in {\cal T}_h} (\lambda{\cal Q}_h(\nabla\cdot\bu),\nabla_w\cdot\bv)_T\\
=& \sum_{T\in {\cal T}_h} -(\bv_0, \nabla(\lambda{\cal Q}_h(\nabla\cdot\bu)))_T+\langle \bv_b, \lambda {\cal Q}_h (\nabla\cdot\bu)\bn\rangle_{\pT}\\
=& \sum_{T\in {\cal T}_h}  ( \nabla\cdot\bv_0, \lambda{\cal Q}_h(\nabla\cdot\bu))_T-\langle \bv_0-\bv_b, \lambda {\cal Q}_h (\nabla\cdot\bu)\bn\rangle_{\pT}
\\
=& \sum_{T\in {\cal T}_h}  ( \nabla\cdot\bv_0, \lambda \nabla\cdot\bu)_T-\langle \bv_0-\bv_b, \lambda {\cal Q}_h (\nabla\cdot\bu)\bn\rangle_{\pT}
\\
=& \sum_{T\in {\cal T}_h}  (\bv_0, -\nabla\cdot(\lambda \nabla\cdot\bu I))_T+\langle \bv_0-\bv_b, \lambda\nabla\cdot\bu\bn\rangle_{\pT}\\
   &\quad 
 -\langle \bv_0-\bv_b, \lambda {\cal Q}_h (\nabla\cdot\bu)\cdot\bn\rangle_{\pT}+\sum_{m=1}^M \langle \bv_b, \ljump\lambda\nabla\cdot\bu   \bn\rjump\rangle_{\Gamma_m}\\
=& \sum_{T\in {\cal T}_h}  (\bv_0, -\nabla\cdot(\lambda \nabla\cdot\bu I))_T -\langle \bv_0-\bv_b, \lambda ( {\cal Q}_h (\nabla\cdot\bu)-\nabla\cdot\bu)\bn\rangle_{\pT}\\
 & \quad +\sum_{m=1}^M \langle \bv_b, \ljump\lambda\nabla\cdot\bu   \bn\rjump\rangle_{\Gamma_m},
\end{split}
\end{equation*}
where we used $\sum_{T\in {\cal T}_h}\langle \bv_b, \ljump\lambda\nabla\cdot\bu  \bn\rjump\rangle_{\pT}=\sum_{m=1}^M \langle \bv_b, \ljump\lambda\nabla\cdot\bu   \bn\rjump\rangle_{\Gamma_m}$ since $\bv_b=0$ on $\partial \Omega$.
 
 Adding the above two equations gives rise to
 \begin{equation*}
 \begin{split}
 &s(\bQ_h\bu, \bv)+\sum_{T\in {\cal T}_h}(2\mu\epsilon_w(\bQ_h\bu), \epsilon_w(\bv))_T+(\lambda\nabla_w\cdot(\bQ_h\bu), \nabla_w\cdot\bv)_T\\
  =&s(\bQ_h\bu, \bv)+\sum_{T\in {\cal T}_h} (\bv_0, -\nabla\cdot(2\mu \epsilon(\bu)))_T
  \\&\quad -\langle 2\mu(\bv_0-\bv_b), ({\cal Q}_h(\epsilon(\bu))-\epsilon(\bu)) \bn \rangle_{\pT} 
    + (\bv_0, -\nabla\cdot(\lambda \nabla\cdot\bu I))_T 
  \\&\quad  -\langle \bv_0-\bv_b, \lambda ({\cal Q}_h (\nabla\cdot\bu)-\nabla\cdot\bu) \bn\rangle_{\pT}+\sum_{m=1}^M\langle \bpsi, \bv_b\rangle_{\Gamma_m}\\
  =& s(\bQ_h\bu, \bv)+(\bf, \bv_0)
  +\sum_{m=1}^M\langle \bpsi, \bv_b\rangle_{\Gamma_m}
  \\&\quad  -\sum_{T\in {\cal T}_h} \langle  (\bv_0-\bv_b), 2\mu({\cal Q}_h(\epsilon(\bu))-\epsilon(\bu)) \bn+\lambda ({\cal Q}_h (\nabla\cdot\bu)
   -\nabla\cdot\bu) \bn \rangle_{\pT}.
  \end{split}
 \end{equation*}
Subtracting \eqref{32} from the above equation yields
\eqref{sehv}.
 \end{proof}

The equation (\ref{sehv}) is called the {\em error
equation} for the WG finite element scheme
(\ref{32}).

\section{Error Estimates}\label{Section:error-estimates}
Note that $\T_h$ is a shape-regular finite element partition of
the domain $\Omega$. For any $T\in\T_h$ and $\varphi\in H^{1}(T)$, the following trace inequality holds true \cite{wy3655}:
\begin{equation}\label{trace-inequality}
\|\varphi\|_{\pT}^2 \leq C
(h_T^{-1}\|\varphi\|_{T}^2+h_T \| \varphi\|_{1, T}^2).
\end{equation}
If $\varphi$ is a polynomial, using the inverse inequality, we have
\begin{equation}\label{trace}
\|\varphi\|_{\pT}^2 \leq Ch_T^{-1}\|\varphi\|_{T}^2.
\end{equation}

\begin{lemma}\label{Lemma5.2}\cite{wy3655} Let ${\cal T}_h$ be a
finite element partition of $\Omega$ satisfying the shape regularity
assumptions as specified in \cite{wy3655}.
The following estimates hold true
\begin{eqnarray}\label{3.2}
\sum_{T\in {\cal T}_h}h_T^{2l}\|\bu-{\cal Q}_h\bu\|^2_{l,T} &\leq &
Ch^{2m}\|\bu\|_{m}^2,\\
\sum_{T\in {\cal T}_h} h_T^{2l}\|\bu- \bQ_0 \bu\|^2_{l,T} &\leq& Ch^{2(m+1)}\|\bu\|_{m+1}^2, \label{3.1}\\
\sum_{T\in {\cal T}_h}h_T^{2l}\|\bu-\bQ_b\bu\|^2_{l,T} &\leq &
Ch^{2m}\|\bu\|_{m}^2, \label{3.3}
\end{eqnarray}
where $0\leq l \leq 2$ and $0\leq m\leq k$.
\end{lemma}

In the weak finite element space $W_h$, we introduce the following semi-norm
$$
\3bar\bv\3bar=\Big(a(\bv, \bv)+s(\bv, \bv)\Big)^{\frac{1}{2}}.
$$

\begin{lemma}
There exists a constant $C$ such that
\begin{equation}\label{eberr}
\Big(\sum_{T\in {\cal T}_h}h_T^{-1}\|\be_0-\be_b\|_{\pT}^2
\Big)^{\frac{1}{2}}\leq C\3bar\be_h\3bar.
\end{equation}
\end{lemma}
\begin{proof}
Using the triangle inequality, the error estimate \eqref{3.3} for the projection $\bQ_b$ and the trace inequality \eqref{trace}, we have
\begin{equation}
\begin{split}
&\Big(\sum_{T\in {\cal T}_h}h_T^{-1}\|\be_0-\be_b\|_{\pT}^2\Big)^{\frac{1}{2}}\\
\leq &\Big(\sum_{T\in {\cal T}_h}h_T^{-1} \|\be_0-\bQ_b\be_0\|_{\pT}^2+h_T^{-1} \|\bQ_b\be_0-\be_b\|_{\pT}^2\Big)^{\frac{1}{2}}\\\leq &\Big(\sum_{T\in {\cal T}_h}h_T^{-2} \|\be_0-\bQ_b\be_0\|_{T}^2+h_T^{-1} \|\bQ_b\be_0-\be_b\|_{\pT}^2\Big)^{\frac{1}{2}}\\
\leq &\Big(\sum_{T\in {\cal T}_h}  |\be_0|_{1,T}^2+h_T^{-1} \|\bQ_b\be_0-\be_b\|_{\pT}^2\Big)^{\frac{1}{2}}\\
\leq& C\3bar \be_h\3bar.
\end{split}
\end{equation}

\end{proof}

The main convergence result can be stated as follows.
\begin{theorem} \label{theoestimate}
Assume $k\geq 1$ and the coefficient coefficients $\mu$ and $\lambda$ are piecewise constants with respect to the finite element partition ${\cal T}_h$. Let $\bu$ and $\bu_h \in W_h$ be the exact solution of the elasticity interface problem \eqref{model} and its numerical solution arising from the WG scheme (\ref{32}). Assume that $\bu$ is sufficiently regular such that $u\in \prod_{i=1}^N  [H^{k+1}(\Omega_i)]^d$. The following error estimate holds true:
 \begin{equation}\label{erres}
\3bar \be_h \3bar \leq Ch^{k}\left( \sum_{i=1}^N \|\bu\|^2_{k+1, \Omega_i} \right)^{\frac{1}{2}}.
\end{equation}
\end{theorem}

\begin{proof} Letting $\bv=\be_h$ in (\ref{sehv}) gives rise to
\begin{equation}\label{t_1}
\begin{split}
&s(\be_h, \be_h)+a(\be_h, \be_h)=s(\bQ_h\bu, \be_h) \\&-\sum_{T\in {\cal T}_h} \langle  \be_0-\be_b, 2\mu({\cal Q}_h(\epsilon(\bu))-\epsilon(\bu))\cdot\bn+\lambda ({\cal Q}_h (\nabla\cdot\bu)-\nabla\cdot\bu)\cdot\bn \rangle_{\pT}.
\end{split}
\end{equation} 
Using the Cauchy-Schwarz inequality, the trace inequality
\eqref{trace-inequality} and the estimate \eqref{3.1}, we have
\begin{equation}\label{t_2}
\begin{split}
&s(\bQ_h\bu, \be_h)\\&=\sum_{T\in {\cal T}_h} \langle \bQ_b(\bQ_0\bu)-\bQ_b\bu, \bQ_b\be_0-\be_b\rangle_{\pT}\\
&\leq \Big(\sum_{T\in {\cal T}_h} h_T\|\bQ_b(\bQ_0\bu)-\bQ_b\bu\|_{\pT}^2\Big)^{\frac{1}{2}} \Big(\sum_{T\in {\cal T}_h}h_T^{-1}\|\bQ_b\be_0-\be_b\|_{\pT}^2\Big)^{\frac{1}{2}}\\
&\leq   \Big(\sum_{T\in {\cal T}_h} h_T\|\bQ_0\bu-\bu\|_{\pT}^2\Big)^{\frac{1}{2}} \Big(\sum_{T\in {\cal T}_h}h_T^{-1}\|\bQ_b\be_0-\be_b\|_{\pT}^2\Big)^{\frac{1}{2}}\\
&\leq   \Big(\sum_{T\in {\cal T}_h} \|\bQ_0\bu-\bu\|_{T}^2+  h_T^2\|\bQ_0\bu-\bu\|_{1, T}\Big)^{\frac{1}{2}} \\ &\quad \cdot
   \Big(\sum_{T\in {\cal T}_h}h_T^{-1}\|\bQ_b\be_0-\be_b\|_{\pT}^2\Big)^{\frac{1}{2}}\\
&\leq Ch^{k+1}\left( \sum_{i=1}^N \|\bu\|^2_{k+1, \Omega_i} \right)^{\frac{1}{2}}\3bar \be_h\3bar.
\end{split}
\end{equation}
Using the Cauchy-Schwarz inequality, the trace inequality \eqref{trace-inequality}, \eqref{eberr} and the estimate \eqref{3.2}, we have
\begin{equation}\label{t_3}
\begin{split}
& \sum_{T\in {\cal T}_h} \langle \be_0-\be_b, 2\mu({\cal Q}_h(\epsilon(\bu))-\epsilon(\bu))\cdot\bn+\lambda ({\cal Q}_h (\nabla\cdot\bu)-\nabla\cdot\bu)\cdot\bn \rangle_{\pT}\\
\leq&\sum_{T\in {\cal T}_h} \langle \be_0-\be_b, 2\mu({\cal Q}_h(\epsilon(\bu))-\epsilon(\bu))\cdot\bn+\lambda ({\cal Q}_h (\nabla\cdot\bu)-\nabla\cdot\bu)\cdot\bn \rangle_{\pT}\\
\leq& \Big(\sum_{T\in {\cal T}_h}h_T^{-1}\|\be_0-\be_b\|_{\pT}^2\Big)^{\frac{1}{2}}
   \\&\quad \cdot
  \Big(\sum_{T\in {\cal T}_h}h_T\|2\mu({\cal Q}_h(\epsilon(\bu))-\epsilon(\bu))\cdot\bn+\lambda ({\cal Q}_h (\nabla\cdot\bu)-\nabla\cdot\bu)\cdot\bn\|_{\pT}^2\Big)^{\frac{1}{2}}\\
\leq& \Big(\sum_{T\in {\cal T}_h}h_T^{-1}\|\be_0-\be_b\|_{\pT}^2\Big)^{\frac{1}{2}} 
  \cdot \Big\{\Big(\sum_{T\in {\cal T}_h}\|2\mu({\cal Q}_h(\epsilon(\bu))-\epsilon(\bu))\cdot\bn\|_{T}^2   \\& +h_T^2\|2\mu({\cal Q}_h(\epsilon(\bu))-\epsilon(\bu))\cdot\bn\|_{1,T}^2\Big)^{\frac{1}{2}} 
  +\Big(\sum_{T\in {\cal T}_h} \|\lambda ({\cal Q}_h (\nabla\cdot\bu)-\nabla\cdot\bu)\cdot\bn\|_{T}^2\\&  +h_T^2\|\lambda ({\cal Q}_h (\nabla\cdot\bu)-\nabla\cdot\bu)\cdot\bn\|_{1,T}^2\Big)^{\frac{1}{2}}\Big\}\\
\leq & Ch^k\3bar\be_h\3bar\left( \sum_{i=1}^N \|\bu\|^2_{k+1, \Omega_i} \right)^{\frac{1}{2}}.
\end{split}
\end{equation}
Substituting \eqref{t_2}-\eqref{t_3} into \eqref{t_1} completes the proof of the theorem. 
\end{proof}

\section{Numerical Experiments}\label{Section:numerical-experiments}
\def\ad#1{\begin{aligned}#1\end{aligned}}  \def\b#1{\mathbf{#1}} \def\hb#1{\hat{\mathbf{#1}}}
\def\a#1{\begin{align*}#1\end{align*}} \def\an#1{\begin{align}#1\end{align}} \def\t#1{\hbox{#1}}
\def\e#1{\begin{equation}#1\end{equation}} \def\div{\operatorname{div}}
\def\p#1{\begin{pmatrix}#1\end{pmatrix}} \def\t#1{\text{#1}}

In this section, we shall report some numerical results for the weak Galerkin finite element scheme \eqref{32} proposed and analyzed for elasticity interface problem \eqref{model} in the previous sections.

In all three test examples, we choose the domain  $\Omega=(0,1)^2$.
 We solve the following linear elasticity interface problem: Find $\bu \in H^1_0(\Omega)^2$
    such that
  \an{\label{w-f} (2\mu\bepsilon(\bu), \bepsilon(\bv)) +(\lambda 
   \nabla \cdot \bu , \nabla \cdot \bv) &=(\bf,\bv) \quad \forall \bv \in H^1_0(\Omega)^2.
  }  The coefficients are set as, for various positive constants $\mu_0$ and $\lambda_0$, 
\an{\label{c-2} \ad{ \mu(x,y)&= \begin{cases} \frac{\mu_0}{2} &(x,y)\in(\frac 14,\frac 34)^2, \\
                                              \frac{1}{2}     & \t{elsewhere,} \end{cases} \\
               \lambda (x,y)&= \lambda_0 \mu(x,y). } }

  \begin{figure}[htb]\begin{center}
\includegraphics[width=2.3in]{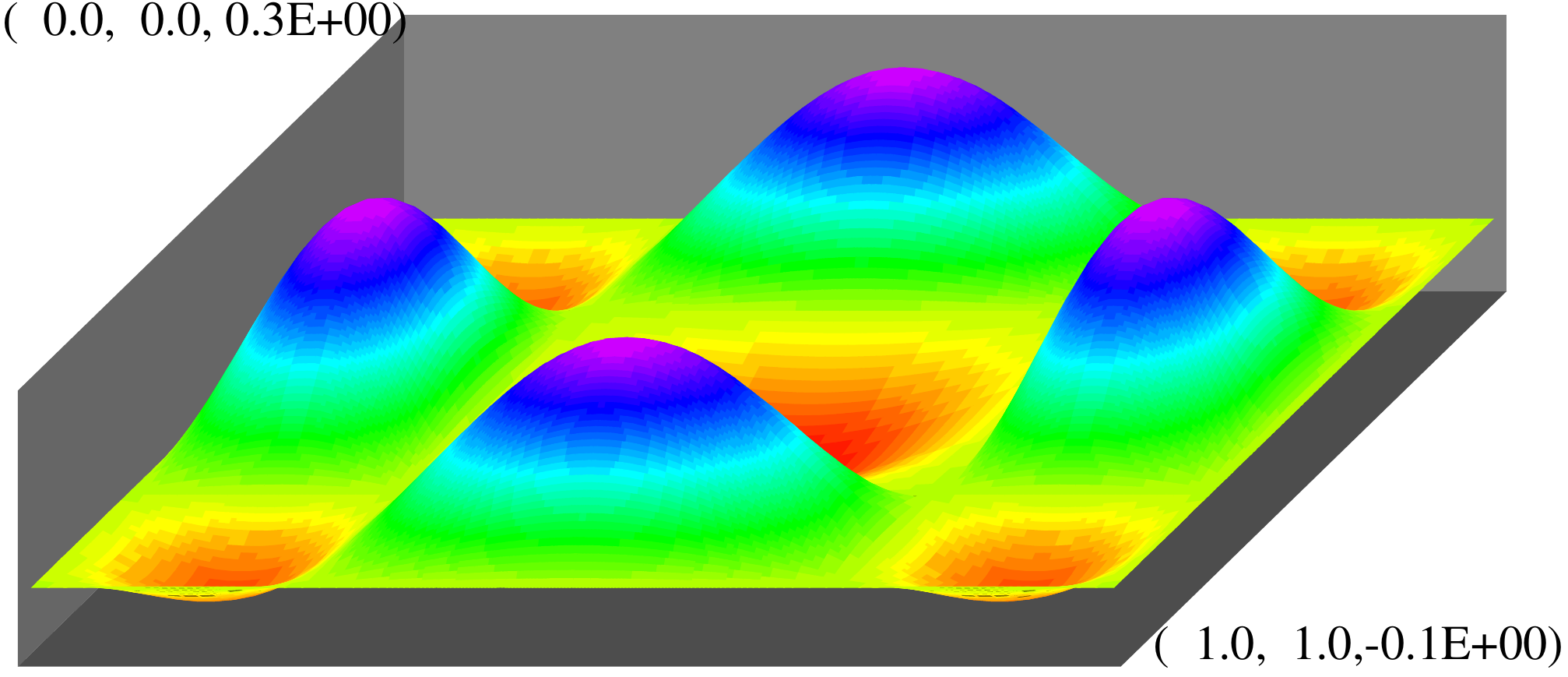} \ 
\includegraphics[width=2.3in]{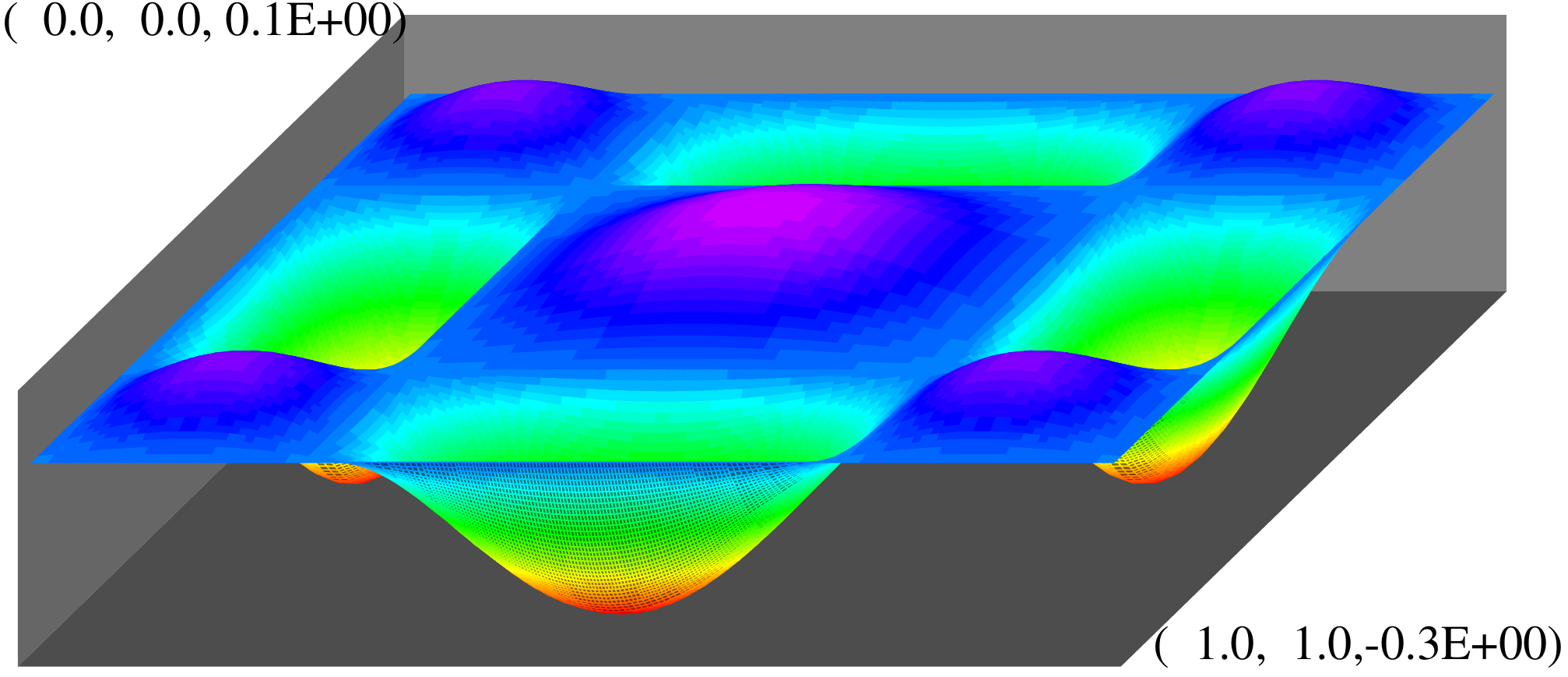}
\caption{The exact solution $\bu$  with $\mu_0=10$ and $\lambda_0=10$.  }
\label{g-s-1}
\end{center}
\end{figure}

\begin{figure}[htb]\begin{center}
\includegraphics[width=1.2in]{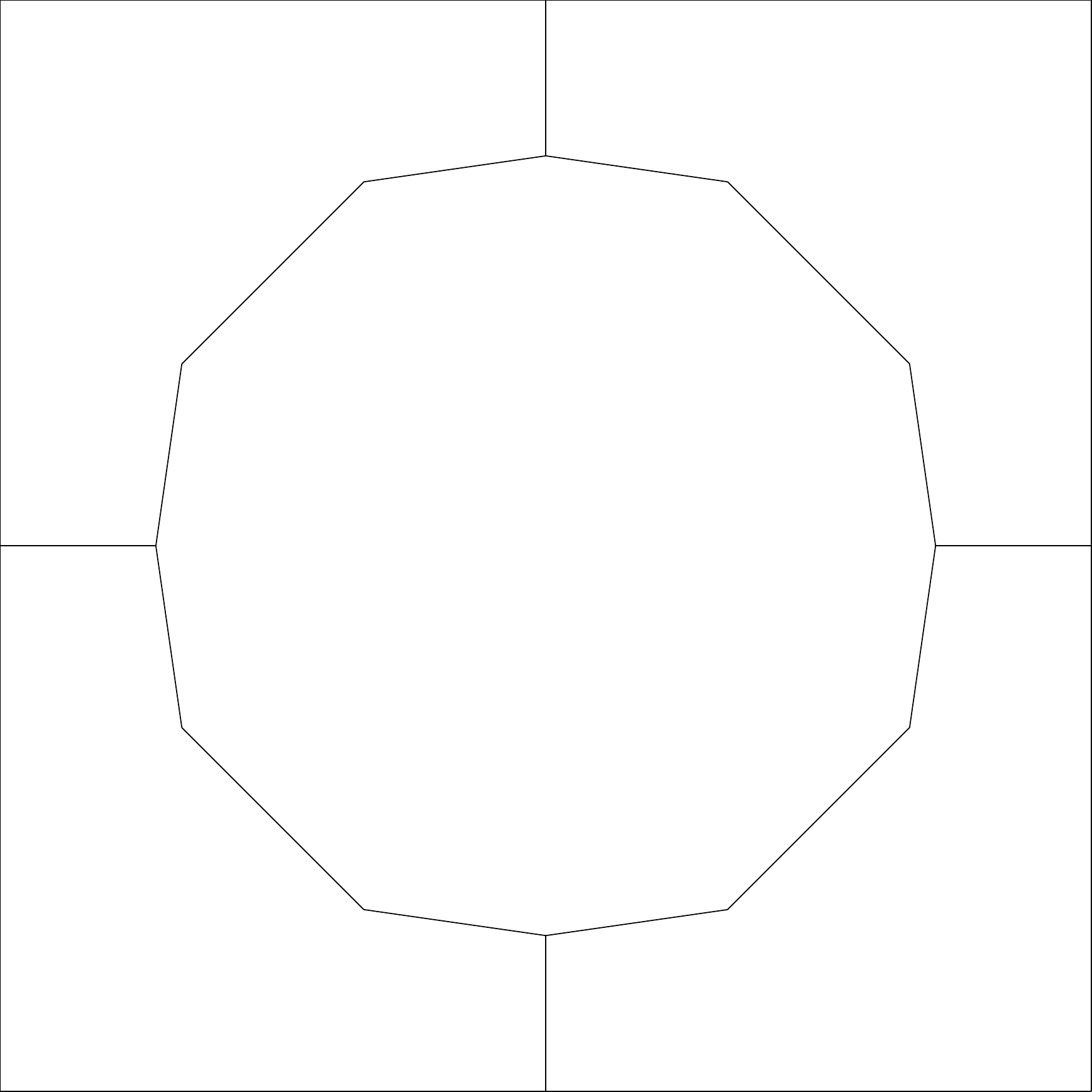} \ 
\includegraphics[width=1.2in]{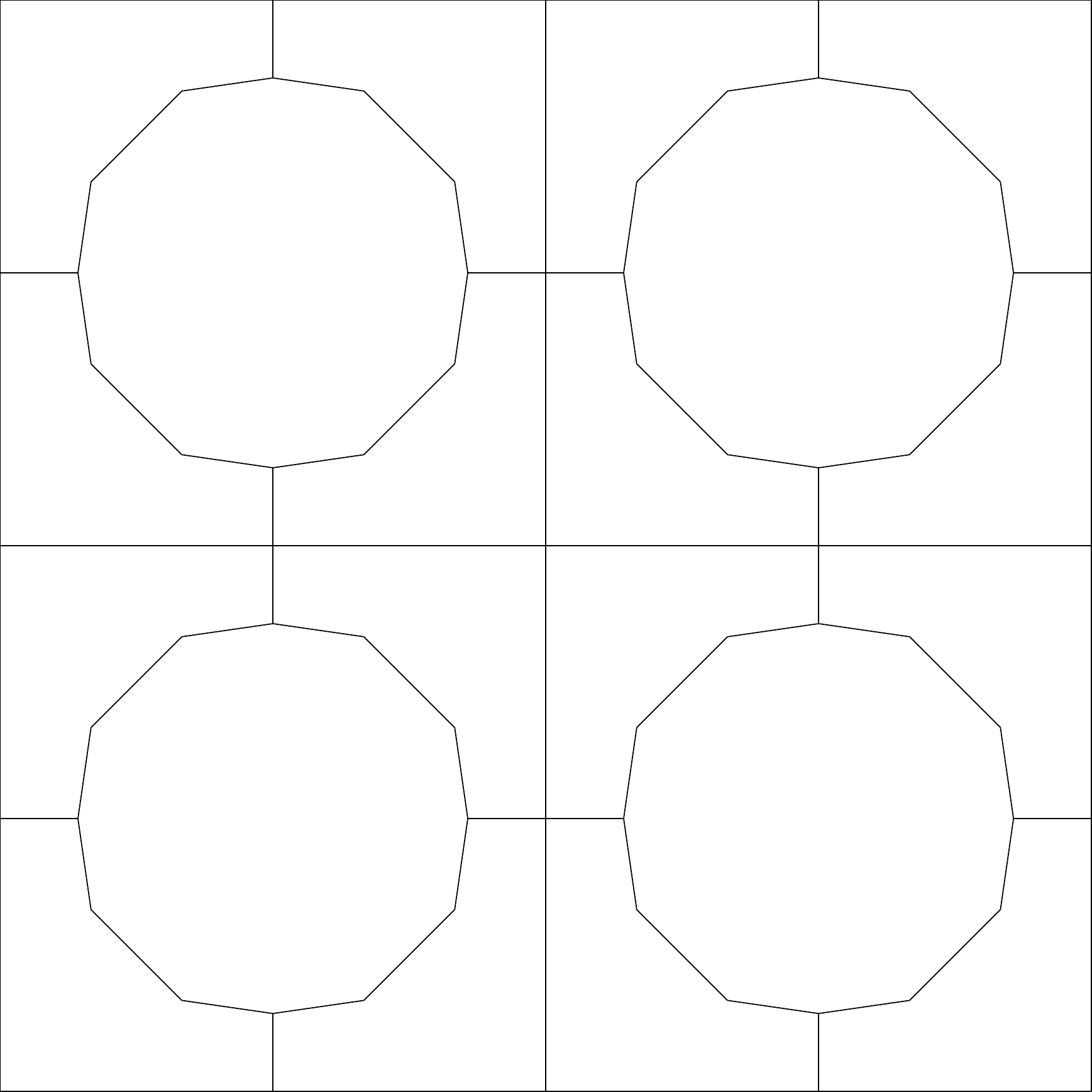} \ 
\includegraphics[width=1.2in]{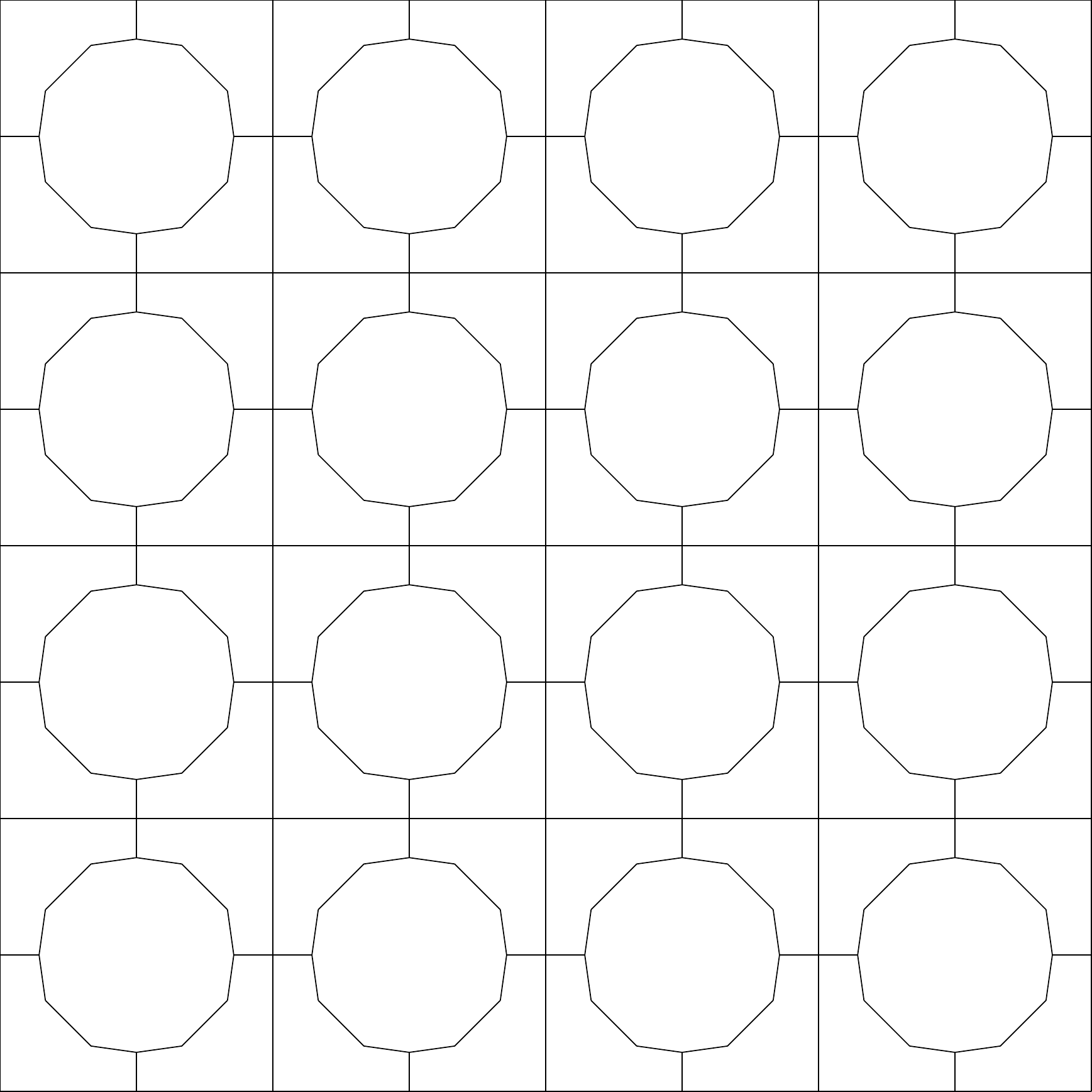}  
\caption{The polygonal meshes  consisting of dodecagons and heptagons are employed in Test Example 1.  }
\label{g-12}
\end{center}
\end{figure}

\subsection{ Test Example 1}  
We carefully choose an exact solution which has a jumping stress at the interface,
\an{\label{s-1} \bu&=\p{-\mu^{-1}\sin(\pi x)\cos(2\pi x)\sin(\pi y)\cos(2\pi y)\\
                      \mu^{-1}\sin(\pi x)\cos(2\pi x)\sin(\pi y)\cos(2\pi y)}.
  } 
The  load function is computed using  the exact solution \eqref{s-1}  and the elasticity interface problem \eqref{model}; i.e.,
\an{\label{f-1}
   \bf&=\left\{ \ad {
  \frac{\pi^2}2 & \Big(  
     36\cos^3\pi x \cos^3 \pi y (2\lambda_0+1)\\ &
    +36 \cos^2 \pi x  \sin \pi y  \cos^2 \pi y  \sin \pi x (2\lambda_0+3) \\ &
    -30 \cos^3 \pi x  \cos \pi y(2\lambda_0+1) 
    -2 \cos^2 \pi x  \sin \pi y   \sin \pi x (18\lambda_0+23) \\ &
    -30 \cos \pi x  \cos^3  \pi y (2\lambda_0+1)  
    -2 \sin \pi x  \sin \pi y^2  \cos \pi y (10\lambda_0+19) \\ &
   +25 \cos \pi x  \cos \pi y (2\lambda_0+1)
    +5\sin \pi x  \sin \pi y(2\lambda_0+3)\Big), \\
4 \pi^2 & \sin \pi x \sin \pi y  (  \cos^2 \pi x - \cos^2\pi y) (2\lambda_0+1)\\ &
  \frac{\pi^2}2   \Big(  
     36\cos^3\pi x \cos^3 \pi y (2\lambda_0+1)\\ &
    +36 \cos^2 \pi x  \sin \pi y  \cos^2 \pi y  \sin \pi x (2\lambda_0+3) \\ &
    -30 \cos^3 \pi x  \cos \pi y(2\lambda_0+1) 
    -2 \cos^2 \pi x  \sin \pi y   \sin \pi x (18\lambda_0+23) \\ &
    -30 \cos \pi x  \cos^3  \pi y (2\lambda_0+1)  
    -2 \sin \pi x  \sin \pi y^2  \cos \pi y (10\lambda_0+19) \\ &
   +25 \cos \pi x  \cos \pi y (2\lambda_0+1)
    +5\sin \pi x  \sin \pi y(2\lambda_0+3)\Big). 
     } \right\} }
Note that the load function $\bf$ is independent of coefficient jumps.
It would ensure a non-smooth solution.
The exact solution $\bu(x,y)$ with $\mu_0=10$ and $\lambda_0=10$ is demonstrated in Figure \ref{g-s-1}.  
The patterned polygonal 
  meshes  consisting of dodecagons and heptagons  shown in Figure \ref{g-12} are employed in the computation. 
A weak finite element space, based on the lowest order  $k=1$; i.e.,
  $$
W_h=\big\{\{\bv_0, \bv_b\}:\bv_0|_T \in [P_1(T)]^2, \bv_b |_T\in P_{RM}(e), \forall T\in {\cal T}_h \big\},
$$ is used in computing this example.
The corresponding element is named $P^1$ WG element.
 Table \ref{t-11} lists the errors in the discrete norm for $\bu_h$ and in the $L^2$ norm for $\bu_0$,  for the $P^1$ WG element in various jump coefficient cases. 
 We have observed from Table \ref{t-11} that our WG method \eqref{32} for the elasticity interface problem \eqref{model} converges 
   at the optimal order of error estimate in the discrete norm for $\bu_h$ 
  independent of jump coefficients, which is consistent with what the theory expects.
 Furthermore, we have also observed our WG method   converges at the optimal order of error estimate for $\bu_0$ in the $L^2$ norm
  independent of jump coefficients.

\begin{table}[ht]
  \centering   \renewcommand{\arraystretch}{1.05}
  \caption{Error profile of $P_1$ WG on polygonal meshes shown in Figure \ref{g-12} for \eqref{s-1}.}
\label{t-11}
\begin{tabular}{c|cc|cc}
\hline
level & $\|Q_0 u-  u_0 \|_0 $  &rate &  $\3bar Q_h u- u_h \3bar $ &rate   \\
\hline
 &\multicolumn{4}{c}{by the $P_1$ WG element ($\mu_0=1$, $\lambda_0=1$)  } \\ \hline
 4&    0.170E-01 &  1.9&    0.144E+01 &  1.0\\
 5&    0.433E-02 &  2.0&    0.722E+00 &  1.0\\
 6&    0.109E-02 &  2.0&    0.361E+00 &  1.0\\
 7&    0.272E-03 &  2.0&    0.181E+00 &  1.0\\
\hline
 &\multicolumn{4}{c}{by the $P_1$ WG element ($\mu_0=10^2$, $\lambda_0=10^{-2}$)  } \\ \hline
 4&    0.108E-01 &  1.9&    0.785E+00 &  0.9 \\
 5&    0.275E-02 &  2.0&    0.395E+00 &  1.0 \\
 6&    0.689E-03 &  2.0&    0.198E+00 &  1.0 \\
 7&    0.172E-03 &  2.0&    0.989E-01 &  1.0 \\
 \hline 
 &\multicolumn{4}{c}{by the $P_1$ WG element ($\mu_0=10^{-1}$, $\lambda_0=10$)  } \\ \hline
 4&    0.106E+00 &  2.0&    0.747E+01 &  1.0 \\
 5&    0.291E-01 &  1.9&    0.375E+01 &  1.0 \\
 6&    0.766E-02 &  1.9&    0.188E+01 &  1.0 \\
 7&    0.196E-02 &  2.0&    0.943E+00 &  1.0 \\
 \hline
\end{tabular}%
\end{table}%

\subsection{Test Example 2}  
To take advantage of superconvergence of weak Galerkin finite 
   elements, and to show its order-two superconvergence (\cite{Al-Taweel,Mu,Ye-lift2,Ye-Stokes}) ,  the exact solution of \eqref{w-f}  is  chosen to be a piecewise polynomial of
   degree 5 as follows 
\an{\label{s-2} \bu&=\p{-\mu^{-1} y (4 y-1) (4 y-3) (4 x-1) (4 x-3)\\
                      \mu^{-1} x (4 y-1) (4 y-3) (4 x-1) (4 x-3)}.
  }   
  In addition, the stabilizer $s(\cdot,\cdot)$ in the WG algorithm \eqref{32} is dropped in order to achieve the superconvergence.
 The load function $\bf$ is computed using the exact solution $\bu$ in \eqref{s-2} and the elasticity interface equation \eqref{w-f}.  
Again, the load $\bf$ is smooth which implies a non-smooth solution $\bu$.

\begin{figure}[ht]
\begin{center}
 \setlength\unitlength{1pt}
    \begin{picture}(260,80)(0,3)
   \def\sqp{\begin{picture}(20,20)(0,0)  
       \multiput(0,0)(20,0){2}{\line(0,1){20}}  \multiput(0,0)(0,20){2}{\line(1,0){20}} 
      \end{picture}}
    \put(0,0){\setlength\unitlength{4pt} \begin{picture}(20,20)(0,0)
       \multiput(0,0)(20,20){1}{\sqp}       \end{picture}}
    \put(90,0){\setlength\unitlength{2pt} \begin{picture}(20,20)(0,0)
       \multiput(0,0)(20,0){2}{\multiput(0,0)(0,20){2}{\sqp}}  \end{picture}}
    \put(180,0){\setlength\unitlength{1pt} \begin{picture}(20,20)(0,0)
       \multiput(0,0)(20,0){4}{\multiput(0,0)(0,20){4}{\sqp}}  \end{picture}} 
    \end{picture}
    \end{center}
\caption{  The first three rectangular grids used in the computation of Table \ref{t-2-1}. }
\label{g-g-s}
\end{figure}
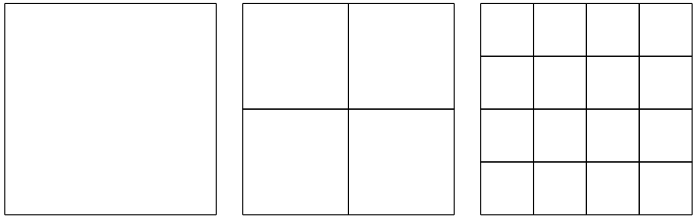

  The uniform rectangular meshes shown in Figure \ref{g-g-s} are employed.  
The numerical results of $P_2$ ($k=2$) and $P_3$ ($k=3$) WG finite elements are listed in
  Table \ref{t-2-1}.
We can see from  Table \ref{t-2-1} that instead of $3$rd order and $2$nd order of convergence  in the $L^2$ norm and the discrete norm for $P_2$ finite element
   solutions, 
     we have 
    $5$th order and and $4$th order of superconvergence in the $L^2$ norm and the discrete  norm  respectively.
Even though the exact solution $\bu$ is a piecewise $P_5$ polynomial, the $P_3$ WG
   finite element solution $\bu_h$ is exactly $Q_h\bu$, if no rounding errors exist.
These results are listed in Table \ref{t-2-1} to see how computer round-off errors behave.

\begin{table}[ht]
  \centering   \renewcommand{\arraystretch}{1.05}
  \caption{ Error profile on rectangular grids
       (Figure \ref{g-g-s}) for \eqref{s-2}. }
\label{t-2-1}
\begin{tabular}{c|cc|cc}
\hline
level & $\|Q_h u-  u_h \|_0 $  &rate &  $\3bar Q_h u- u_h \3bar $ &rate   \\
\hline
 &\multicolumn{4}{c}{by the $P_2$ WG element ($\mu_0=1$, $\lambda_0=1$) } \\ \hline
 3&    0.117E-01 &  5.0&    0.507E+00 &  4.0 \\
 4&    0.366E-03 &  5.0&    0.320E-01 &  4.0 \\
 5&    0.116E-04 &  5.0&    0.209E-02 &  3.9 \\
 6&    0.389E-06 &  4.9&    0.150E-03 &  3.9 \\
\hline
 &\multicolumn{4}{c}{by the $P_3$ WG element ($\mu_0=1$, $\lambda_0=1$) }  \\ \hline 
 3&    0.128E-12 &  0.0&    0.199E-11 &  0.0 \\
 4&    0.519E-12 &  0.0&    0.677E-11 &  0.0 \\
 5&    0.235E-11 &  0.0&    0.248E-10 &  0.0 \\
 \hline
 &\multicolumn{4}{c}{by the $P_2$ WG element ($\mu_0=10^2$, $\lambda_0=10^{-2}$) } \\ \hline
 3&    0.101E-01 &  5.2&    0.366E+00 &  4.2 \\
 4&    0.314E-03 &  5.0&    0.229E-01 &  4.0 \\
 5&    0.982E-05 &  5.0&    0.143E-02 &  4.0 \\
 6&    0.307E-06 &  5.0&    0.893E-04 &  4.0 \\
\hline
 &\multicolumn{4}{c}{by the $P_3$ WG element ($\mu_0=10^2$, $\lambda_0=10^{-2}$) } \\ \hline
 3&    0.847E-13 &  0.0&    0.157E-11 &  0.0 \\
 4&    0.407E-12 &  0.0&    0.613E-11 &  0.0 \\
 5&    0.180E-11 &  0.0&    0.236E-10 &  0.0 \\
 \hline
\end{tabular}%
\end{table}%

\subsection{Test Example 3}
 In this test, we shall compute the numerical solutions arsing from WG scheme \eqref{32} for the   elasticity interface model problem \eqref{w-f} when the exact solution is unknown. 
That is,  we choose a smooth load $\bf$ directly.   Find $\bu \in H^1_0(\Omega)^2$
  such that
 \an{\label{w-f3}  
  ( 2 \mu\bepsilon(\bu), \bepsilon(\bv) +(\lambda 
   \nabla \cdot \bu , \nabla \cdot \bv) &=(\p{0\\ 1} ,\bv) 
   \quad \forall \bv \in H^1_0(\Omega)^2, 
 } where $\Omega=(0,1)^2$, $\mu$ and $\lambda$ are defined in \eqref{c-2}.
The rectangular meshes shown in Figure \ref{g-g-s} are employed. The $P_4$ (k=4) weak Galerkin solution on the 4th rectangular grids shown as in
   Figure \ref{g-g-s}, is plotted in Figure \ref{g-s-3},
with $\mu_0=1$ and $\lambda_0=1$.

\begin{figure}[htb]\begin{center}
\includegraphics[width=2.3in]{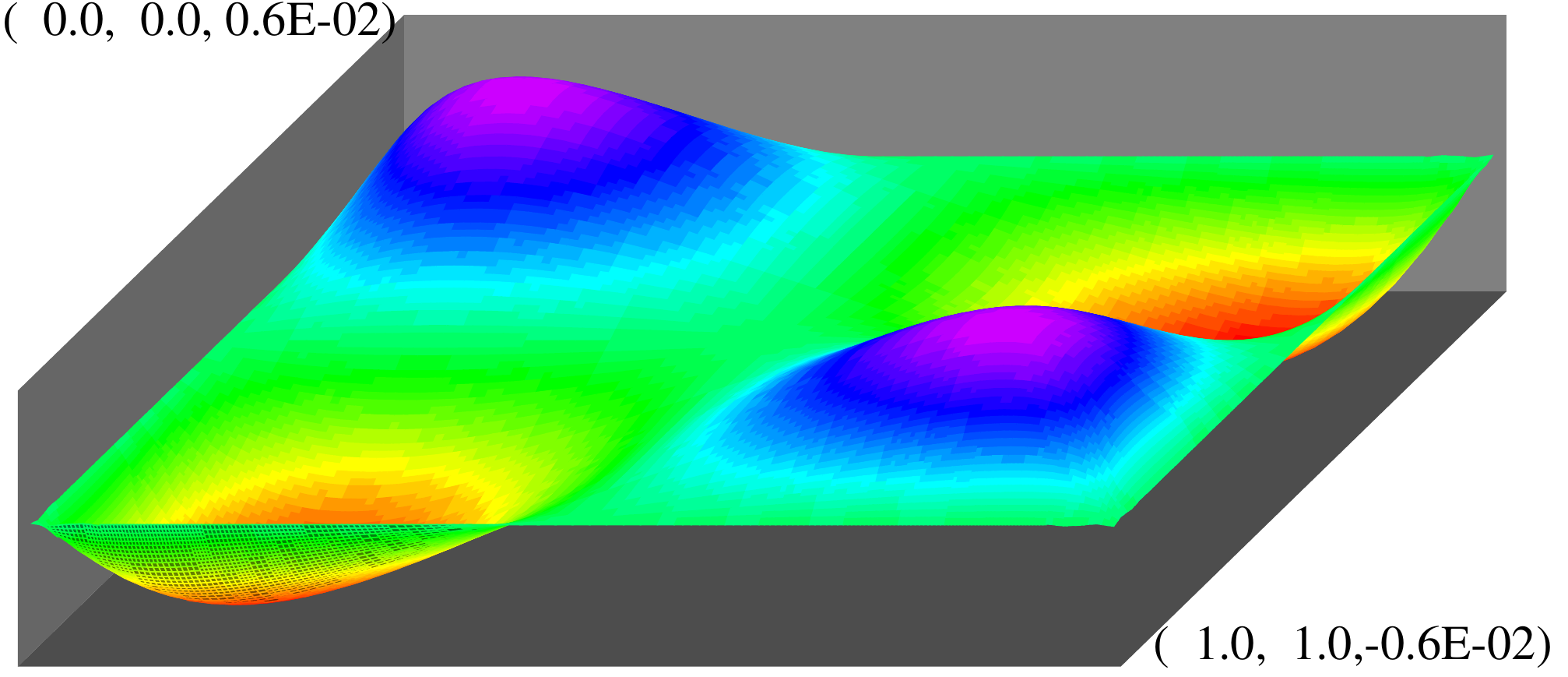} \  
\includegraphics[width=2.3in]{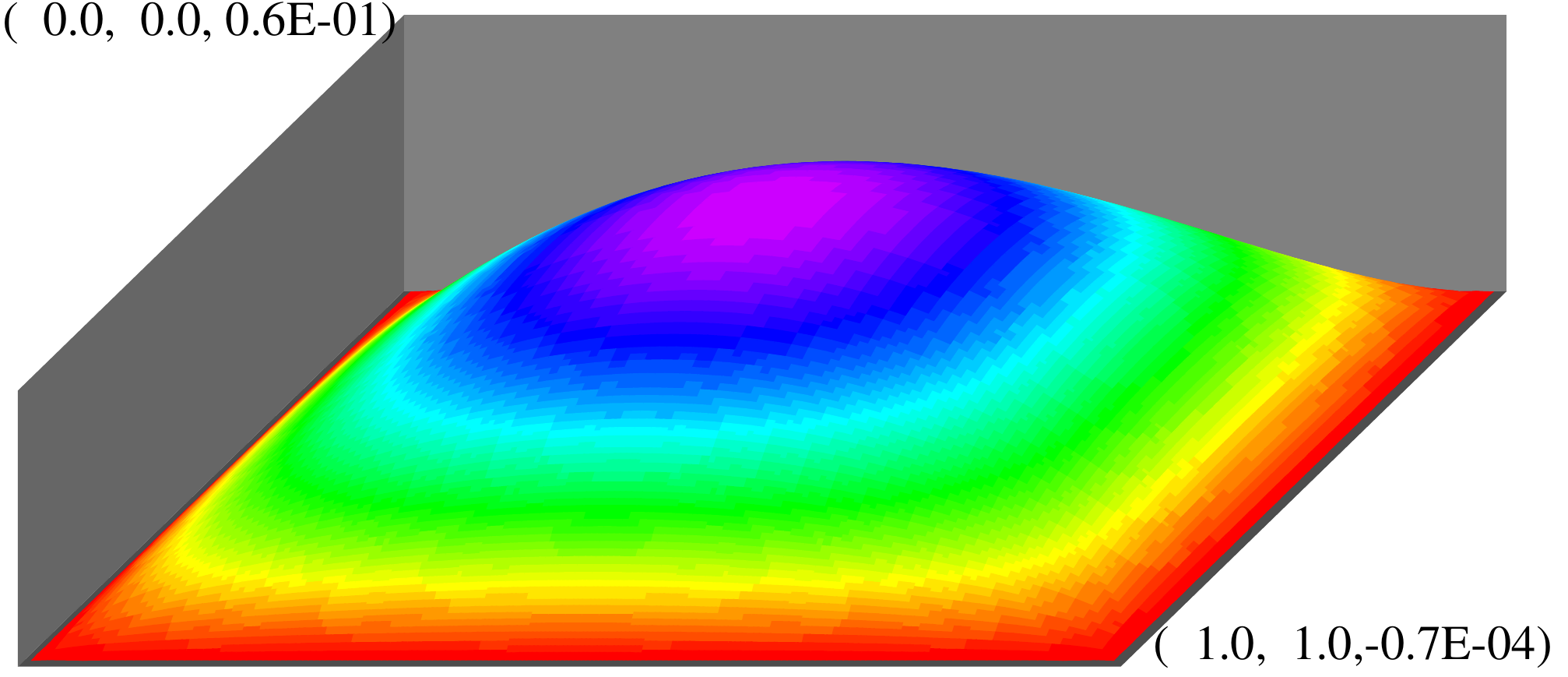}
\caption{The $P_4$ WG solution   ($\mu_0=1$, $\lambda_0=1$)
    on the 4th rectangular grids (cf. Figure \ref{g-g-s}).  }
\label{g-s-3}
\end{center}
\end{figure}

We plot the solutions of the other two interface cases in Figures \ref{g-s-31} and \ref{g-s-32}.
In these two cases, due to re-entrant interface corners,  the solutions are not 
   smooth near these 4 points.
We can observe from 
   Figures \ref{g-s-31} and \ref{g-s-32} that the large jumps of the discontinuous finite element solutions arise at the four points. Additionally we plot the deformation of the linear elasticity for the two cases, under the unit
  upward load,  in Figure \ref{g-v}.
On the left, as we have a hard material in the center, the deformation is almost identical 
  in the center square.  As we have the whole material highly compressible, the deformation 
  is almost one directional, i.e.,  the deformation field is irrotational.
On the right,  as we have the softer material in the center, the deformation is almost limited
  inside the center square.  Here we have nearly incompressible material everywhere.
  The deformation is nearly solenoidal, i.e., the deformation almost
   circles around inside the center   square.

\begin{figure}[htb]\begin{center}
\includegraphics[width=2.3in]{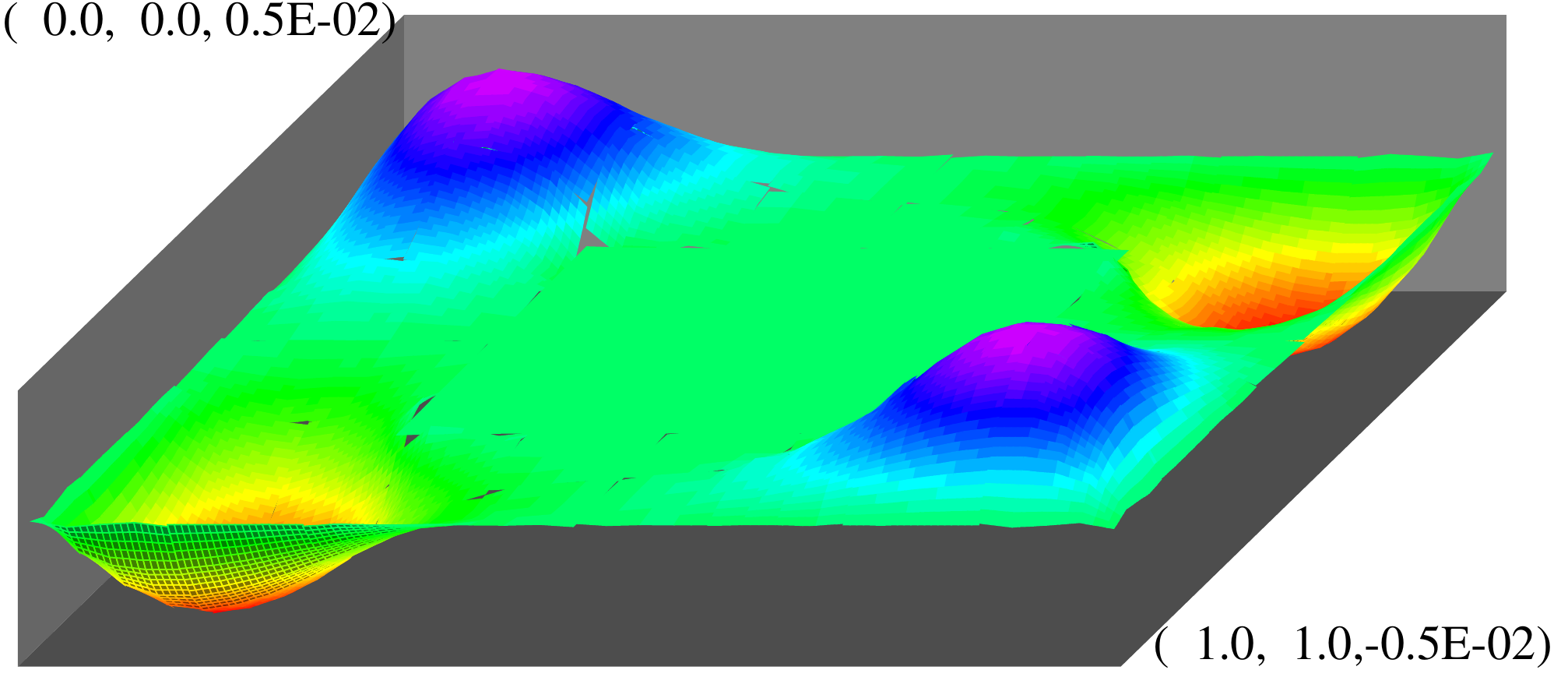} \ 
\includegraphics[width=2.3in]{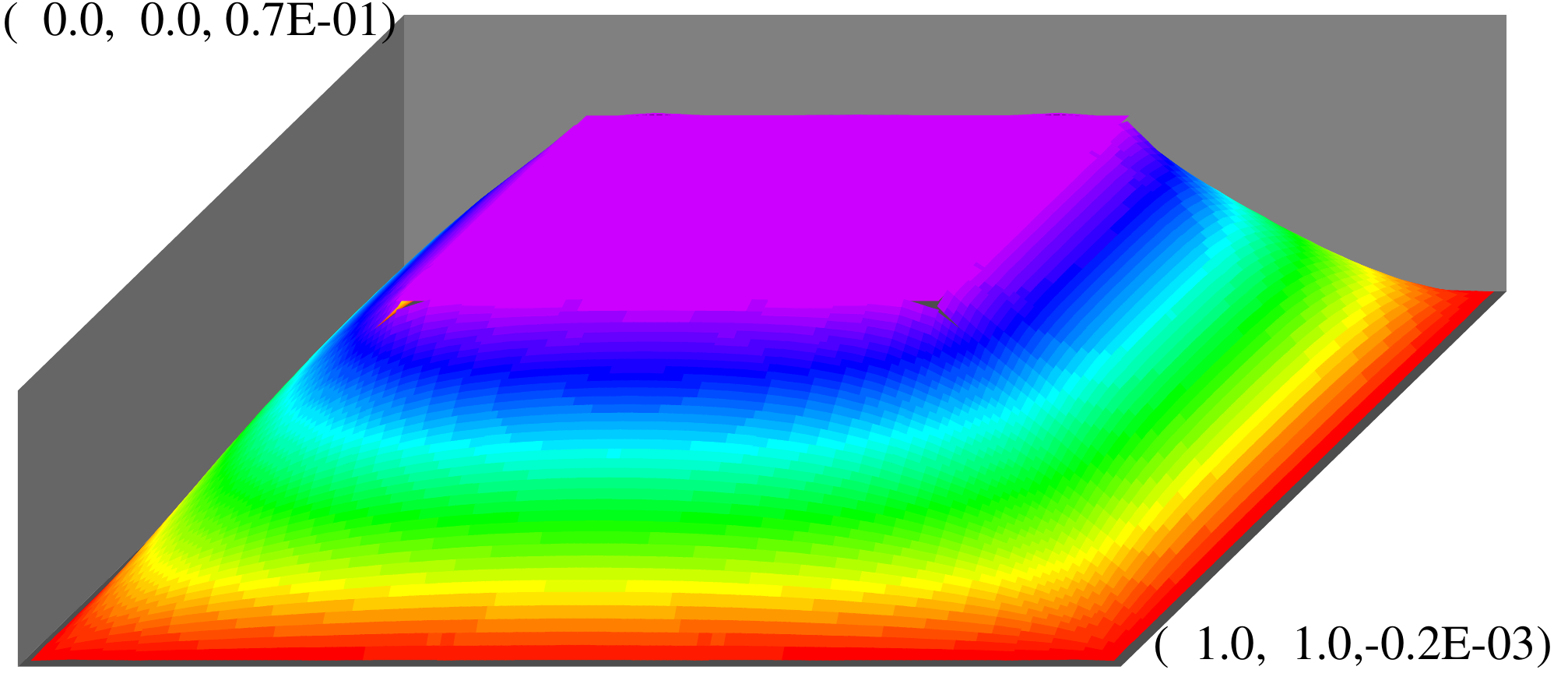}
\caption{The $P_4$ WG solution  ($\mu_0=10^3$, $\lambda_0=10^{-3}$)
    on the 4th rectangular grid (cf. Figure \ref{g-g-s}).  }
\label{g-s-31}
\end{center}
\end{figure}

\begin{figure}[htb]\begin{center}
\includegraphics[width=2.3in]{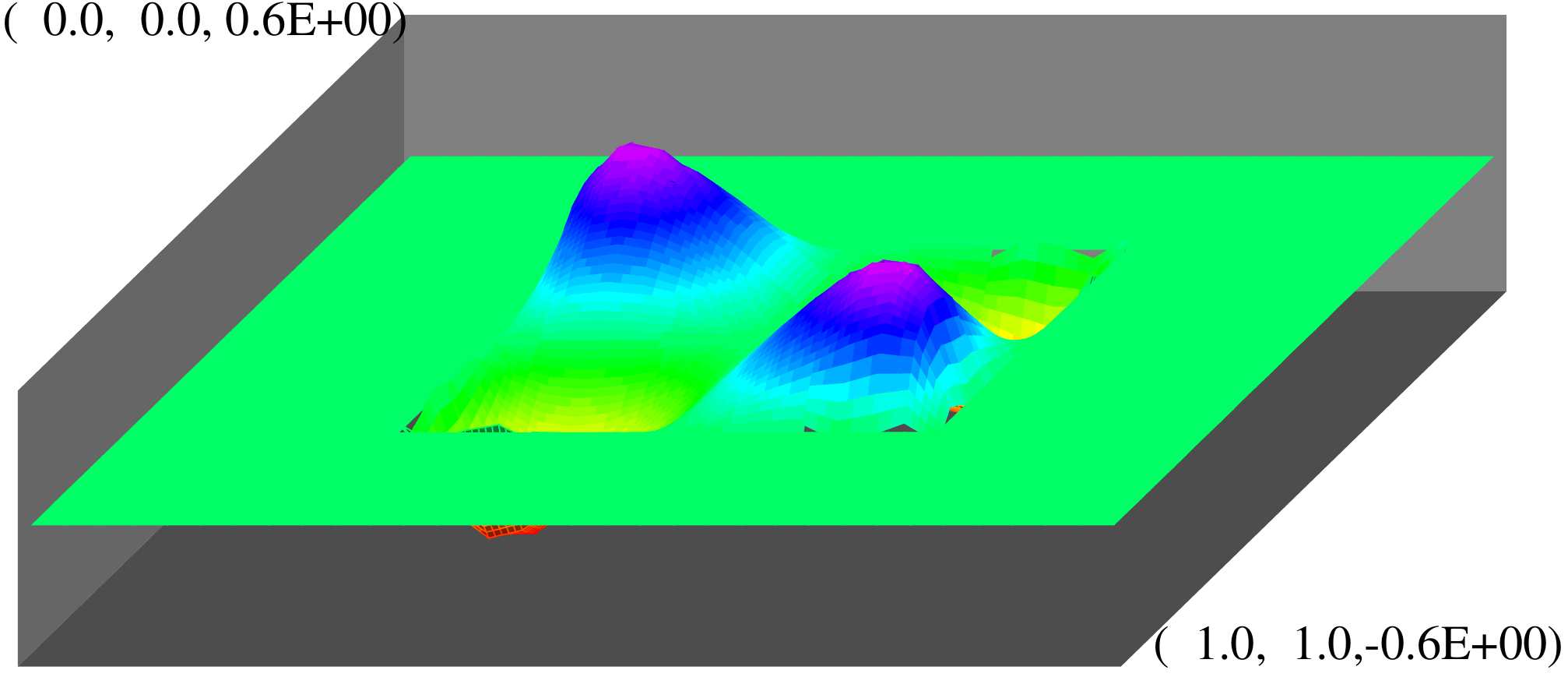}\    
\includegraphics[width=2.3in]{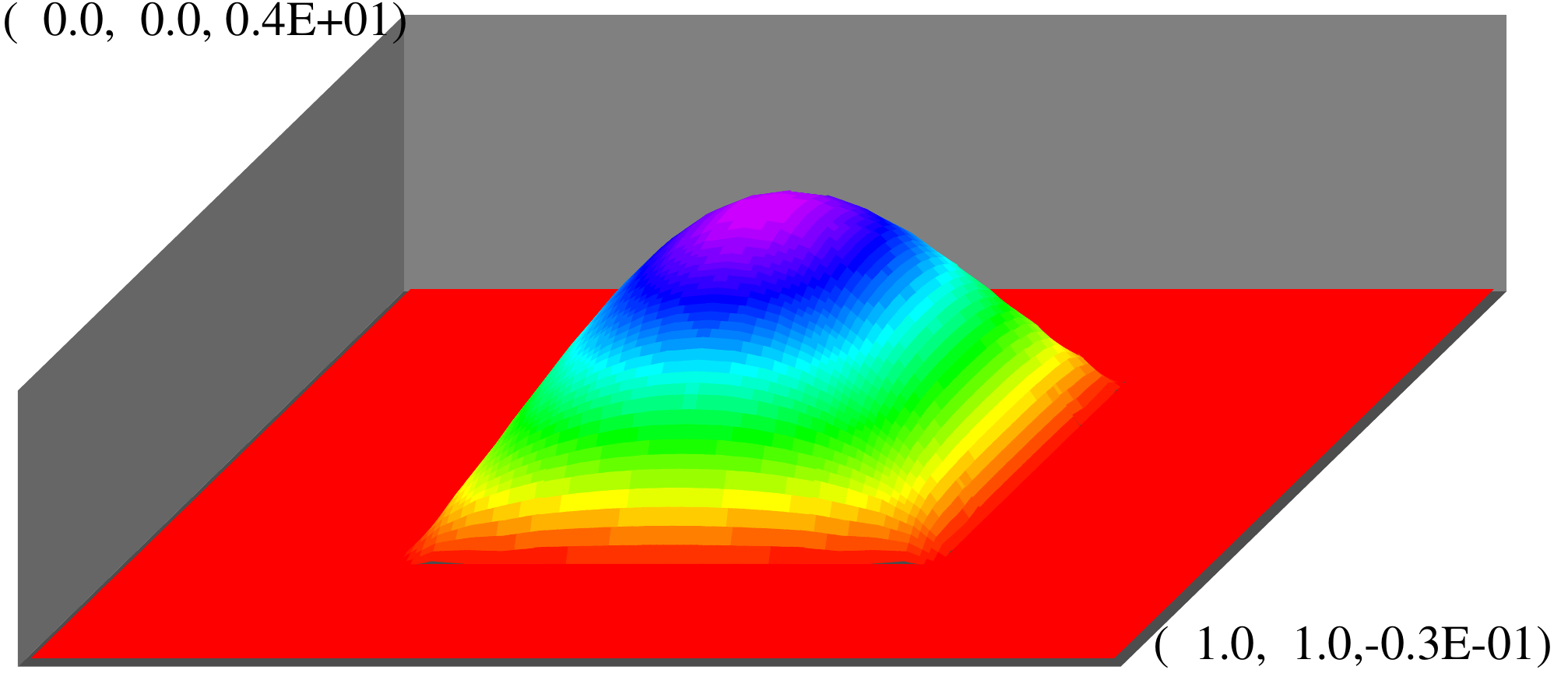}
\caption{The $P_4$ WG solution  ($\mu_0=10^{-3}$, $\lambda_0=10^{3}$)
    on the 4th rectangular grid (cf. Figure \ref{g-g-s}).  }
\label{g-s-32}
\end{center}
\end{figure}

\begin{figure}[htb]\begin{center}
\includegraphics[width=2.1in]{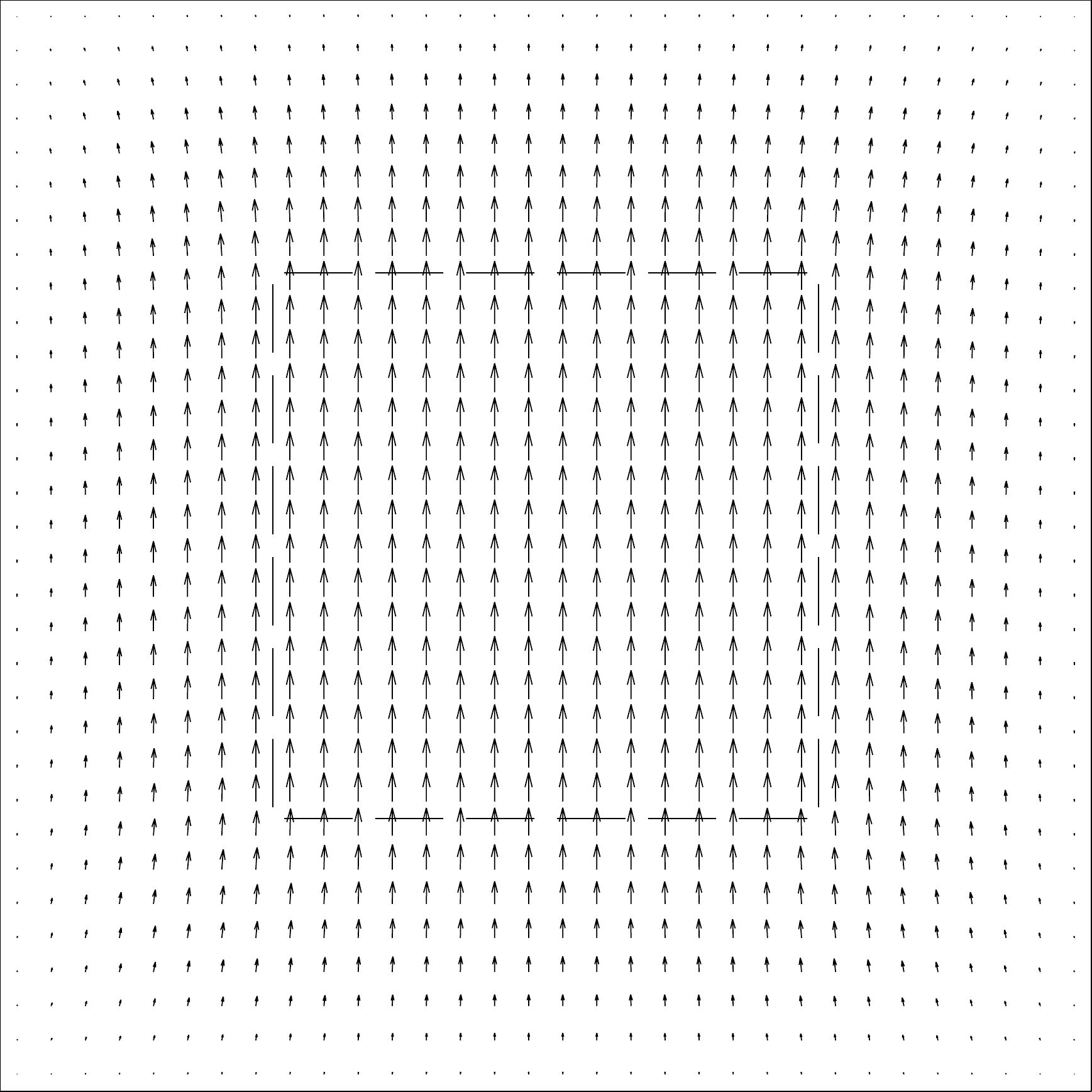}\quad
\includegraphics[width=2.1in]{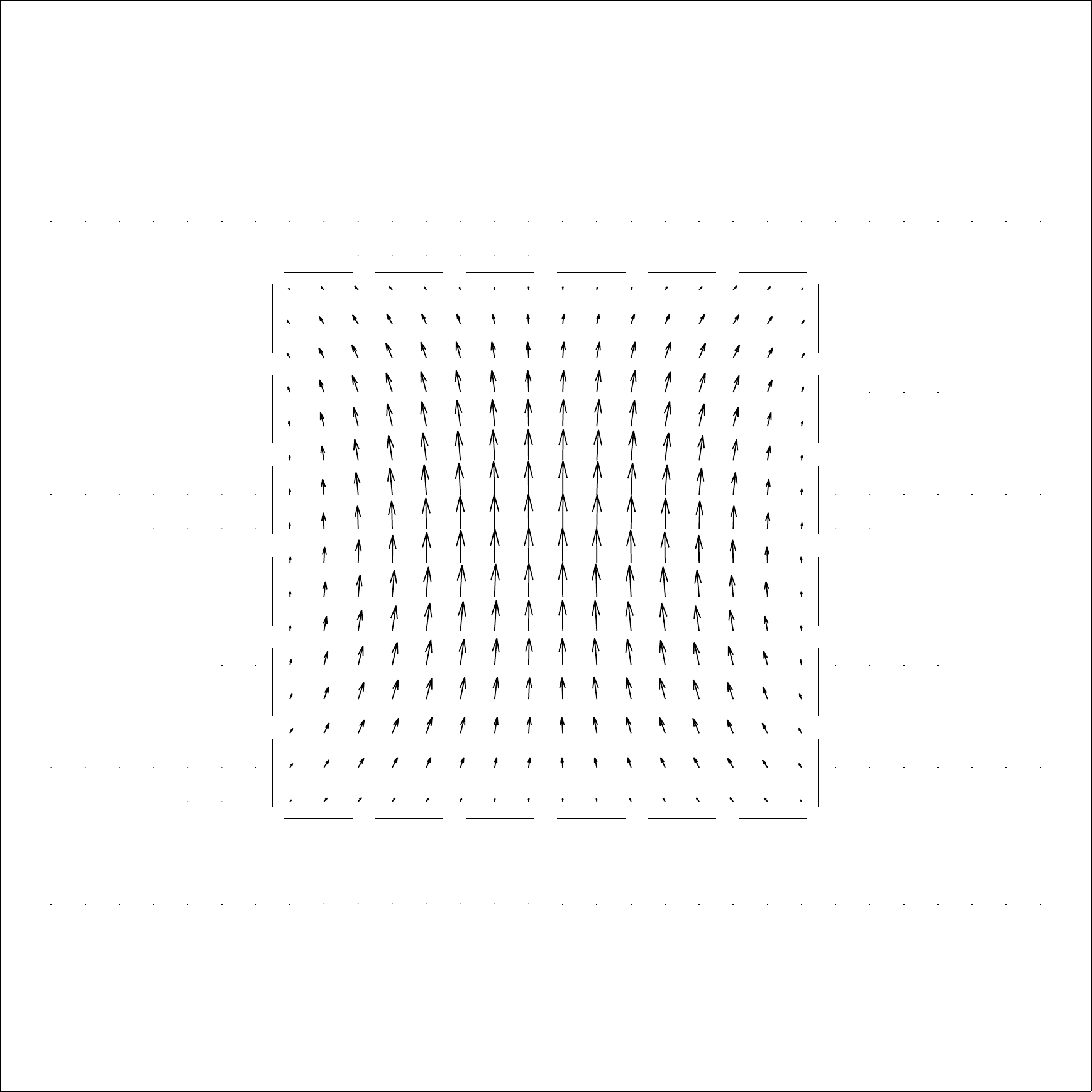} 
\caption{The deformation of a linear elasticity under the unit upward load. \  
     Left: $\mu_0=10^3$, $\lambda_0=10^{-3}$. \ 
    Right: $\mu_0=10^{-3}$, $\lambda_0=10^{3}$.  }
\label{g-v}
\end{center}
\end{figure}

\bibliographystyle{abbrv}
\bibliography{Ref}


\end{document}